\documentclass[11pt,leqno]{amsart}
\usepackage{amsfonts,amscd}
\usepackage[T1]{fontenc} % Output font encoding for international characters

\usepackage{hyperref}
\usepackage[english]{babel}

\theoremstyle{plain}
\newtheorem{theorem}                {Theorem}      [section]
\newtheorem{proposition}  [theorem]  {Proposition}
\newtheorem{corollary}    [theorem]  {Corollary}

\theoremstyle{definition}

\newtheorem{remark}       [theorem]  {Remark}

\numberwithin{equation}{section}

\def \R{{\mathbb R}}
\def \s{{\mathbb S}}

\def \h{{\mathbb H}}

\newcommand{\Ric}{\operatorname{Ric}}
\newcommand{\Id}{\operatorname{Id}}
\newcommand{\Scal}{\operatorname{Scal}}

\usepackage{color}

\def \1 {\`}

\DeclareMathOperator{\grad}{grad}
\DeclareMathOperator{\trace}{Tr}
\numberwithin{equation}{section}

\title[Conformal-biharmonic hypersurfaces]{
 Conformal-biharmonic hypersurfaces in spheres and product spaces}

\author{V.~Branding}
\address{ Institute of Mathematics\\
University of Rostock\\
Ulmenstraße 69, 18057 Rostock, Germany}
\email{volker.branding@uni-rostock.de}

\author{S.~Montaldo}
\address{Dipartimento di Matematica e Informatica\\
Universit\`a degli Studi di Cagliari\\
Via Ospedale 72\\
09124 Cagliari, Italy}
\email{montaldo@unica.it}

\author{S.~Nistor}
\address{Faculty of Mathematics\\ ``Al.I. Cuza'' University of Iasi\\
Bd. Carol I no. 11 \\
700506 Iasi, Romania}
\email{nistor.simona@ymail.com}

\author{C.~Oniciuc}
\address{Faculty of Mathematics\\ ``Al.I. Cuza'' University of Iasi\\
Bd. Carol I no. 11 \\
700506 Iasi, Romania}
\email{oniciucc@uaic.ro}

\author{A.~Ratto}
\address{Dipartimento di Matematica e Informatica\\
Universit\`a degli Studi di Cagliari\\
Via Ospedale 72\\
09124 Cagliari, Italy}
\email{rattoa@unica.it}

\begin{document}

\begin{abstract}

The conformal-bienergy functional $E_2^c$ is a modified version of the classical bienergy functional $E_2$ and it is conformally invariant in the case of a four-dimensional domain. 
The critical points of $E_2^c$ are called conformal-biharmonic and denoted $c$-biharmonic. In the first part of the paper we study the $c$-biharmonic hypersurfaces $M^m$ with constant principal curvatures in the product space $ {\mathbb L}^m(\varepsilon) \times \mathbb{R} $, where $ {\mathbb L}^m(\varepsilon) $ denotes a space form of constant sectional curvature $ \varepsilon $. Specifically, we demonstrate that $ M^m $ is either totally geodesic or a cylindrical hypersurface of the form $ M^{m-1} \times \mathbb{R} $, where $ M^{m-1} $ is an iso\-parametric $c$-biharmonic hypersurface in $ {\mathbb L}^m(\varepsilon) $. In the second part of this article we obtain a full description of isoparametric $c$-biharmonic hypersurfaces in $\s^{m+1}$ and a complete classification of $c$-biharmonic hypersurfaces with constant scalar curvature in $\s^{m+1}$, $m=2,3$ and $m=4$ with an additional assumption. In this context, we shall also prove a global result for compact $c$-biharmonic immersions in $\s^5$. In the final part of the paper, as a preliminary effort to understand $c$-biharmonic hypersurfaces in $ {\mathbb L}^m(\varepsilon) \times \mathbb{R} $ with \textit{non-constant} mean curvature, we establish that a totally umbilical  $c$-biharmonic hypersurface must necessarily be totally geodesic.
\end{abstract}


\thanks{The author V.B. gratefully acknowledges the support of the Austrian Science Fund (FWF)
through the project "Geometric Analysis of Biwave Maps" (DOI: 10.55776/P34853). The authors S.M. and A.R. are members of the Italian National Group G.N.S.A.G.A. of INdAM. The work was partially supported by the Project {PRoBIKI} funded by Fondazione di Sardegna.}

\maketitle
\section{Introduction}

For smooth maps  \(\phi:(M^m,g)\to(N^n,h)\) between two Riemannian manifolds we define the \emph{bienergy functional} by
\begin{align}
\label{bienergy}
E_2(\phi):=\frac{1}{2}\int_M|\tau(\phi)|^2 \ v_g\,,
\end{align}
where 
\begin{align*}
\tau(\phi):=\trace\bar\nabla d\phi, \qquad \tau(\phi)\in C\left(\phi^{-1}TN\right),
\end{align*}
represents the {\em tension field of the map} and \(\bar\nabla\) is the connection on the pull-back bundle \(\phi^{-1}TN\). Solutions of $\tau(\phi)=0$ are called {\em harmonic maps}, while the critical points of \eqref{bienergy} are called \emph{biharmonic maps} and can be expressed by the vanishing of the \emph{bitension field} \(\tau_2(\phi)\) as follows:
\begin{align}
\label{bitension}
0=\tau_2(\phi):=-\bar\Delta\tau(\phi)-\trace R^N(d\phi(\cdot), \tau(\phi))d\phi(\cdot)\,,
\end{align}
where $\bar\Delta$ is the rough Laplacian defined on sections of the pull-back bundle \(\phi^{-1}TN\) and $R^N$
is the curvature operator on $(N^n,h)$.
Biharmonic maps have received a lot of attention in recent years and for the current status of research on biharmonic maps we refer to the book \cite{MR4265170}, while for biharmonic hypersurfaces one may consult \cite{MR4410183}.

The bienergy functional \eqref{bienergy} is invariant under isometries of the domain.
However, regarding the invariance under conformal transformations, it turns out that the bienergy \eqref{bienergy} does not enjoy conformal invariance in any dimension. To repair this flaw, V.Bérard has considered in \cite{MR2449641,MR2722777} (see also \cite{MR3028564} for more details), for each even dimensional domain, a modified version of the energy functional called {\it conformal-energy}. The critical points of these energies, called \textit{conformal-harmonic maps}, are solutions of a nonlinear elliptic PDE of order equal to $m=\dim M$ and invariant under conformal changes of the metric on $M$.  When $m=4$ the B\'erard functional coincides with the bienergy \eqref{bienergy} with two additional terms as follows (see \cite{MR2449641}):
\begin{align}
\label{bienergy-conformal}
E^c_2(\phi):=\frac{1}{2}\int_M \left(|\tau(\phi)|^2+\frac{2}{3}\Scal |d\phi|^2-2\trace\langle d\phi\left(\Ric(\cdot)\right), d\phi(\cdot)\rangle\right) \ v_g\,,
\end{align}
where \(\Scal\) is the scalar curvature of \(M\) and \(\Ric\) is its Ricci tensor.

Consequently, this functional can be viewed as the natural higher-order generalization of the classical energy functional, which is conformally invariant when \(\dim M=2\). The structure of \eqref{bienergy-conformal} is inspired by the Paneitz operator, a fourth-order elliptic operator in conformal geometry. For details, see \cite[Chapter 4]{MR2521913} and the original work \cite{MR2393291}.

Following \cite{2501.08804, arXiv:2311.04493} we call \eqref{bienergy-conformal} the \emph{conformal-bienergy functional} and, although it is conformally invariant only when $\dim M=4$, we shall study its geometrical properties for any dimension of the domain.
The critical points of the conformal bienergy \eqref{bienergy-conformal} are those who satisfy (see  \cite{MR2722777} and also \cite{arXiv:2311.04493})
\begin{align}
\label{eq-$c$-biharmonic-intro}
0=\tau_2^c(\phi) & := \tau_2(\phi)-\frac{2}{3}\Scal\tau(\phi)+2\trace(\bar \nabla d\phi)(\Ric(\cdot),\cdot)+\frac{1}{3}d\phi(\nabla\Scal).
\end{align}
Solutions of \eqref{eq-$c$-biharmonic-intro} will be called \emph{conformal-biharmonc maps}, or simply 
\emph{$c$-biharmonic maps}.\\

In the present paper we are interested in \textit{$c$-biharmonic hypersurfaces}, that is \textit{isometric immersions} $\varphi: M^m\hookrightarrow (N^{m+1},h)$\textit{ which are solutions of} \eqref{eq-$c$-biharmonic-intro}. \\

When $(N^{m+1},h)$ is a space form ${\mathbb L}^{m+1}(\varepsilon)$ with constant sectional curvature $\varepsilon$ a study of $c$-biharmonic hypersurfaces with constant principal curvatures was initiated in \cite{arXiv:2311.04493}. In particular, they proved that in the case $\varepsilon=0$ (${\mathbb L}^{m+1}(0)=\R^{m+1}$) the only examples are the totally geodesic planes. When $\varepsilon=-1$ (${\mathbb L}^{m+1}(-1)={\mathbb H}^{m+1}$) they gave a detailed analysis of  $c$-biharmonic hypersurfaces with at most two distinct constant principal curvatures. It should be noted that in space forms a hypersurface has constant principal curvatures if and only if it is isoparametric. Moreover, if the ambient space is ${\mathbb H}^{m+1}$, a classical argument of Cartan \cite{MR1553310} shows that an isoparametric hypersurface has at most two distinct principal curvatures, thus the study in \cite{arXiv:2311.04493} covers all the isoparametric $c$-biharmonic  hypersurfaces in ${\mathbb H}^{m+1}$. As for the case of isoparametric hypersurfaces in ${\mathbb S}^{m+1}$, in \cite{arXiv:2311.04493} the authors gave a full description of those with $1$ or $2$ principal curvatures and an example with $4$ distinct principal curvatures.\\

When $(N^{m+1},h)$ is the product ${\mathbb L}^{m}(\varepsilon)\times \R$ a hypersurface with constant principal curvatures is not necessarily isoparametric. In fact it is known (see \cite{MR4032814}) that hypersurfaces with constant principal curvatures are isoparametric provided that the angle $\alpha$ between the normal to the hypersurface and the $\R$ direction of the ambient space is constant. Recently, in \cite{arXiv:2411.11506}, the authors proved that an isoparametric hypersurface in ${\mathbb L}^{m}(\varepsilon)\times \R$, $\varepsilon=\pm 1$, has constant angle. Thus we have that isoparametric hypersurfaces in ${\mathbb L}^{m}(\varepsilon)\times \R$, $\varepsilon=\pm 1$, are equivalent to hypersurfaces with constant principal curvatures and constant angle.\\

Our first contribution is the following theorem which characterizes $c$-biharmo\-nic hypersurfaces with constant principal curvatures in the product ${\mathbb L}^{m}(\varepsilon)\times \R$ without the assumption that the angle $\alpha$ is constant.

\begin{theorem}\label{theorem:cbiharmonicconstantcurvatures}
Let $M^m$ be a $c$-biharmonic hypersurface with constant principal curvatures in the product ${\mathbb L}^{m}(\varepsilon)\times \R$. Then $M^m$ is totally geodesic or $M^m=M^{m-1}\times\R$, where $M^{m-1}$ is an isoparametric $c$-biharmonic hypersurface in ${\mathbb L}^{m}(\varepsilon)$.
\end{theorem}

Consequently, the complete understanding of $c$-biharmonic hypersurfaces with constant principal curvatures in the product space ${\mathbb L}^{m}(\varepsilon) \times \R$ reduces to a comprehensive study of $c$-biharmonic hypersurfaces with constant principal curvatures in the space form ${\mathbb L}^m(\varepsilon)$. As previously mentioned, $c$-biharmonic hypersurfaces with constant principal curvatures in ${\mathbb L}^m(\varepsilon)$ have been fully classified for $\varepsilon \in \{-1, 0\}$, whereas only partial results are available for $\varepsilon = 1$. 

For this reason, in Section~\ref{isoparametric-in-sm} we shall provide a complete study of isoparametric $c$-biharmonic  hypersurfaces in ${\mathbb S}^{m+1}$ considering all possible degrees (i.e., the number of distinct principal curvatures) $\ell = 1, 2, 3, 4, 6$. 

The study of Section~\ref{isoparametric-in-sm} will be used in Section~\ref{sec-new-remark} to prove the following three classification results for $c$-biharmonic immersions $\varphi: M^m \hookrightarrow \s^{m+1}$, $m=2,3,4$. These results are stated separately because the method of proof depends on $m$.

\begin{theorem}\label{teo2-m=2-remark}
Let $\varphi: M^2 \hookrightarrow \s^3$ be a $c$-biharmonic surface. If $M^2$ has constant scalar curvature then, up to isometries, it is an open subset of one of the following:
\begin{enumerate}
\item[(i)] the totally geodesic $\s^2$ or the small hypersphere $\s^2\left(\frac{1}{\sqrt{3}}\right)$
\item[(ii)] the minimal Clifford torus 
$$
\s^1\left(\frac{1}{\sqrt 2}\right) \times \s^1\left(\frac{1}{\sqrt 2}\right)\,.
$$
\end{enumerate}
\end{theorem}  

\begin{theorem}\label{teo2-remark}
Let $\varphi: M^3 \hookrightarrow \s^4$ be a $c$-biharmonic hypersurface. If $M^3$ has constant scalar curvature then, up to isometries, it is an open subset of one of the following:
\begin{enumerate}
\item[(i)] the totally geodesic $\s^3$ or the small hypersphere $\s^3\left(\frac{1}{\sqrt{2}}\right)$
\item[(ii)] the extrinsic product 
$$
\s^1\left(r_1\right) \times \s^2\left(r_2\right)\,,\quad r_1^2=\frac{T^*}{1+T^*}\,,\; r_2^2=\frac{1}{1+T^*}\,,
$$
where $T^*$ is the unique positive root of the polynomial 
$$
8  T^3-32  T^2+11  T-3
$$
\item[(iii)] the minimal isoparametric hypersurface of degree $3$.
\end{enumerate}
\end{theorem}

\begin{theorem}\label{teo2-m=4-remark}
Let $\varphi: M^4 \hookrightarrow \s^5$ be a $c$-biharmonic hypersurface. If $M^4$ has constant scalar curvature and at most three distinct principal curvatures at any point then, up to isometries, it is an open subset of one of the following:
\begin{enumerate}
\item[(i)] the totally geodesic $\s^4$ or the small hypersphere $\s^4\left(\frac{\sqrt 3}{2}\right)$
\item[(ii)] the extrinsic product 
$$
\s^2\left(r_1\right) \times \s^2\left(r_2\right)
$$
where either
$$
r_1^2=r_2^2=\frac{1}{2}\quad {\rm or}\quad r_1^2=\frac{1}{2}\left (1-\frac{1}{\sqrt 3} \right )\quad {\rm and}\quad r_2^2=\frac{1}{2}\left (1+\frac{1}{\sqrt 3} \right )\,;
$$
the extrinsic product 
$$
\s^1\left(r_1\right) \times \s^3\left(r_2\right)\,,\quad r_1^2=\frac{T^*}{1+T^*}\,,\quad r_2^2=\frac{1}{1+T^*}\,,
$$
where $T^*$ is the unique positive root of the third order polynomial
\begin{equation}\label{eq:P3(T)}
P_3(T)=9 T^3-19  T^2+3 T-1\,.
\end{equation}
\end{enumerate}
\end{theorem}

\begin{remark} In Theorem~\ref{th-isop-minimal} we shall prove the existence of a minimal, isoparametric $c$-biharmonic hypersurface $\varphi: M^4 \hookrightarrow \s^5$ of degree $\ell=4$, i.e., with four distinct principal curvatures. 
\end{remark}

Theorems~\ref{teo2-m=2-remark}, \ref{teo2-remark} and \ref{teo2-m=4-remark} have a \textit{local} nature. An interesting feature of the $c$-biharmonic tension field is the fact that in the case of hypersurfaces it contains the term ${\rm Tr}A^3$. Combining this fact with a recent result of Z. Tang and W. Yan (see \cite{MR4529031}) we shall prove the following \textit{global} property:
\begin{theorem}\label{teo-global}
Let $\varphi: M^4 \hookrightarrow \s^5$ be a $c$-biharmonic hypersurface. Assume that $M^4$ is compact, CMC and of nonnegative constant scalar curvature. Then, up to isometries, $\varphi: M^4 \hookrightarrow \s^5$ is either one of the examples listed in Theorem~\ref{teo2-m=4-remark} or the minimal isoparametric hypersurface of degree $4$. 
\end{theorem}

Finally, as an initial step towards understanding $c$-biharmonic hypersurfaces in \( {\mathbb L}^m(\varepsilon) \times \mathbb{R} \) with non-constant principal curvatures, we examine \textit{totally umbilical} hypersurfaces in \( {\mathbb L}^m(\varepsilon) \times \mathbb{R} \). These are hypersurfaces characterized by having a single principal curvature, denoted \( H \), with multiplicity \( m \). In the case of hypersurfaces in a space form, it is straightforward to show that totally umbilical hypersurfaces must have constant \( H \). By contrast,  the mean curvature $H$ of a totally umbilical hypersurface in \( {\mathbb L}^m(\varepsilon) \times \mathbb{R} \)  is not necessarily  a constant (see, for instance, \cite{MR4493653}).  
For this class of hypersurfaces, we obtain the following result:

\begin{theorem}\label{teo-rigidity-totally-geodesic}
Let $M^m\hookrightarrow {\mathbb L}^m(\varepsilon)\times \R$ be a totally umbilical $c$-biharmonic hypersurface. Then $M^m$ is totally geodesic. 
\end{theorem}

The paper is organized as follows: in Section~\ref{sec-preliminaries} we carry out some preliminary computations on $c$-biharmonic hypersurfaces in $\mathbb{ L}^m(\varepsilon)\times \R$. In Section~\ref{sec-proofTh1} we prove Theorem~\ref{theorem:cbiharmonicconstantcurvatures}. Next, in Section~\ref{sec-iso-Sn} we study in detail isoparametric $c$-biharmonic hypersurfaces in the Euclidean sphere $\s^{m+1}$. In Section~\ref{sec-new-remark} we shall prove Theorems~\ref{teo2-m=2-remark}, \ref{teo2-remark}, \ref{teo2-m=4-remark} and \ref{teo-global}. Finally, in the last section we provide the proof of Theorem~\ref{teo-rigidity-totally-geodesic}.\\

{\small {\em Notations}. Throughout this article we will use the following sign conventions: 
for the Riemannian curvature tensor field we use 
$$
R(X,Y)Z=[\nabla_X,\nabla_Y]Z-\nabla_{[X,Y]}Z,
$$ 
for the Ricci tensor field 
$$
g(\Ric(X),Y)=\Ric(X,Y)=\trace \left\{Z\to R(Z,X)Y\right\},
$$
and the scalar curvature is given by
$$
\Scal=\trace\Ric.
$$
The trace is taken with respect to the domain metric and we write $\trace$ instead of $\trace_g$.

For the rough Laplacian on the pull-back bundle $\phi^{-1} TN$ we employ the geometers sign convention
$$
\bar\Delta = -\trace(\bar\nabla\bar\nabla-\bar\nabla_\nabla).
$$
Finally, throughout the paper we shall assume that all the hypersurfaces $M^m$ are oriented and without boundary.

\section{$c$-biharmonic hypersurfaces in ${\mathbb L}^m(\varepsilon)\times \R$} \label{sec-preliminaries}
 
Let ${\mathbb L}^m(\varepsilon)$ represent, depending on the value of $\varepsilon$, the space form $\s^m$, $\h^m$ or $\R^m$ with
constant curvature $\varepsilon=1$, $-1$ or $0$, respectively.
 
Given a hypersurface  
\[
\varphi:M^m\hookrightarrow {\mathbb L}^m(\varepsilon)\times \mathbb R\,,
\]
let $\eta$ denote a unit vector field normal to
$\varphi$ and let $\partial_t$ denote a unit vector field tangent to the second factor $\R$ of the ambient space. Let us denote by $\langle,\rangle$ the product metric in ${\mathbb L}^m(\varepsilon)\times \R$. We further denote by $A$ the {\em shape operator}, by $H=\trace A/m$ the {\em mean curvature function} and by $B(X,Y)=\langle A(X),Y\rangle \eta$ its {\em second fundamental form}, $X,Y\in C(TM^m)$.

Since $\partial_t$ is a unit vector field globally defined on the
ambient space ${\mathbb L}^m(c)\times \mathbb R$, we can decompose it in the
following form
\begin{equation}
\label{eq:dt}
\partial_t=d\varphi(T)+\cos\alpha\,\eta,
\end{equation}
where $\cos\alpha= \langle \partial_t, \eta\rangle$  and  $d\varphi(T), T\in C(TM^m)$, denotes the tangential component of $\partial_t$. In the sequel we shall identify $T$ with $d\varphi(T)$.

Using that $\partial_t$ is parallel on ${\mathbb L}^m(c)\times \mathbb R$, a direct
computation yields the useful formulas 
\begin{eqnarray}
\nabla_XT&=&\cos\alpha \,A(X)\label{eq:nablaXT}\\
X(\cos\alpha)&=&-\langle A(X),T\rangle \label{eq:XTheta}
\end{eqnarray}
for every tangent vector field $X\in C(TM^m)$.

The Riemannian curvature tensor of ${\mathbb L}^m(\varepsilon)\times\mathbb{R}$ can be easily computed and it is given
by the formula
\begin{align}\label{eq:tildeR}
\bar R(\bar X,\bar Y)\bar Z=\varepsilon\{\langle \bar Y,\bar Z\rangle \bar X-\langle \bar X,\bar Z\rangle \bar Y -\langle
\bar Y,\partial_t\rangle\langle \bar Z,\partial_t\rangle \bar X
+\langle \bar X,\partial_t\rangle\langle \bar Z,\partial_t\rangle \bar Y\nonumber\\
+\langle \bar X,\bar Z\rangle\langle \bar Y,\partial_t\rangle \partial_t -\langle
\bar Y,\bar Z\rangle\langle \bar X,\partial_t\rangle \partial_t\},
\end{align}
where $\bar X, \bar Y, \bar Z$ are vector fields on ${\mathbb L}^m(\varepsilon)\times\mathbb{R}$.

Finally, the Codazzi equation is given by
\begin{eqnarray}
\label{eq:Codazzieq} 
(\nabla_XA)Y-(\nabla_YA)X = \varepsilon
\cos\alpha(\langle Y,T\rangle X-\langle X,T\rangle Y),
\end{eqnarray}
where $X,Y\in C(T M^m)$.\\

We are in the right position to give, in terms of the angle function $\alpha$,  the
$c$-biharmonic tension field \eqref{eq-$c$-biharmonic-intro} of a hypersurface  $\varphi:M^{m}\hookrightarrow {\mathbb L}^m(\varepsilon)\times \mathbb{R}$.

\begin{proposition} Let $\varphi:M^{m}\hookrightarrow {\mathbb L}^m(\varepsilon)\times \mathbb{R}$  be a hypersurface. Then the  $c$-biharmonic tension field is given by
$$
\tau_2^c=\left[\tau_2^c\right]^{\top}+\left[\tau_2^c\right]^{\perp}
$$
where
\begin{equation}\label{BHP1}
\begin{aligned}
\left[\tau_2^c\right]^{\top}= & -2 m A( \grad H)-m^2 H \grad H-2 m(m-1) \varepsilon H \cos \alpha\, T \\
& -\frac{4}{3}(m-1) \varepsilon \cos \alpha A(T)-\frac{1}{3} \grad |A|^2+\frac{2}{3} m^2 H \grad H
\end{aligned}
\end{equation}
is the tangential component, while
\begin{equation}\label{BHP2}
\begin{aligned}
\left[\tau_2^c\right]^{\perp} = &\big\{-m \Delta H+\frac{2}{3} m(m-1)(3-m) \varepsilon H+\frac{1}{3} m(7 m-13) \varepsilon H \sin^2 \alpha \\
&+\frac{5}{3} m H |A|^2+2(m-2) \varepsilon T(\cos\alpha) -2 \trace A^3-\frac{2}{3} m^3 H^3 \big\}\eta
\end{aligned}
\end{equation}
is the normal component.
\end{proposition}

\begin{proof}We proceed by computing all terms in \eqref{eq-$c$-biharmonic-intro} and then summing them up.
The tension field of the hypersurface is $\tau(\phi)=m H \eta$ and the bitension field, as computed in \cite{MR4333889}, is expressed as follows:
\begin{align}\label{eq0-comp-ctau2}
\tau_2(\phi)=&-2m A(\grad H)-m^2 H \grad H - 2 \varepsilon m (m-1) \cos\alpha H T\nonumber\\
&+ m \big\{-\Delta H-  H |A|^2 + \varepsilon  (m-1) H \sin^2\alpha\big\}\eta.
\end{align}
Let $\{E_i\}_{i=1}^m$ be a local orthonormal frame field on $M$. According to the Gauss equation, we have:
\begin{align}\label{eq1-comp-ctau2}
\Ric(X,Y)=\sum_i \langle R(X,E_i)E_i,Y\rangle =& \sum_i \langle \bar R(X,E_i)E_i,Y\rangle\nonumber \\
&+\sum_i \left( \langle B(X,Y),B(E_i,E_i)\rangle-\langle B(X,E_i),B(Y,E_i)\rangle\right)\,.
\end{align}
Now, using \eqref{eq:tildeR}, we compute 
\begin{align}\label{eq2-comp-ctau2}
\sum_i \langle \bar R(X,E_i)E_i,Y\rangle= \varepsilon\langle (m-1)X-|T|^2 X-(m-2)\langle X,T \rangle T,Y\rangle
\end{align}
and, by the definition of $B$,
\begin{align}\label{eq3-comp-ctau2}
\sum_i \langle B(X,Y),B(E_i,E_i)\rangle&= m H\langle  A(X),Y \rangle \\
-\sum_i \langle  B(X,E_i),B(Y,E_i)\rangle&=- \langle A^2(X),Y \rangle\,.
\end{align}
From this, we derive:
\begin{align}\label{eq4-comp-ctau2}
\Ric(X)=\varepsilon(m-1)X-\varepsilon|T|^2 X-\varepsilon(m-2) \langle X,T \rangle T-A^2(X)+mHA(X) \,.
\end{align}
Furthermore, the scalar curvature can be computed as follows:
\begin{align}\label{eq5-comp-ctau2}
\Scal=\sum_i \langle \Ric(E_i), E_i\rangle=\varepsilon m(m-1)-2\varepsilon(m-1) |T|^2 -|A|^2+m^2H^2\,.
\end{align}
Next we compute the term $2\trace(\bar \nabla d\phi)(\Ric(\cdot),\cdot)$. 
We have, taking into account \eqref{eq4-comp-ctau2} and using \eqref{eq:Codazzieq},
\begin{align}\label{eq6-comp-ctau2}
 2\sum_i(\bar \nabla d\phi)(\Ric(E_i),E_i)=&2 \varepsilon (m-1) \sum_i(\bar \nabla d\phi)(E_i,E_i)
 -2 \varepsilon |T|^2 \sum_i(\bar \nabla d\phi)(E_i,E_i)\nonumber\\
 & -2 \varepsilon(m-2) \sum_i(\bar \nabla d\phi)(T,\langle E_i, T\rangle E_i)- 2  \sum_i(\bar \nabla d\phi)(A^2(E_i),E_i)\nonumber\\
 & + 2 m H  \sum_i(\bar \nabla d\phi)(A(E_i),E_i)\nonumber\\
 =& 2 \varepsilon m (m-1) H \eta - 2 \varepsilon m |T|^2  H \eta - 2 \varepsilon (m-2) B(T,T)\nonumber\\
 &-2 \sum_i B(A^2(E_i),E_i)+2 m H \sum_i B(A(E_i),E_i)\nonumber\\
 =&\big\{ 2 \varepsilon m(m-1) H -2 \varepsilon m |T|^2 H+ 2 \varepsilon (m-2) T(\cos\alpha)\nonumber\\
 & -2 \trace A^3+2m H |A|^2\big\}\eta\,.
\end{align}
Finally, taking into account \eqref{eq:nablaXT}, we have 
$$
E_i (|T|^2) = 2  \langle A(E_i),T\rangle \cos\alpha\,,
$$
and using \eqref{eq5-comp-ctau2} we obtain
\begin{align}\label{eq7-comp-ctau2}
\frac{1}{3} d\phi (\nabla\Scal)=&\frac{1}{3} \sum_i (E_i\Scal)E_i\nonumber\\
=&-\frac{4}{3} \varepsilon (m-1) \cos\alpha A(T)-\frac{1}{3} \grad |A|^2+\frac{2}{3} m^2 H \grad H\,.
\end{align}
By summing up all the computed terms in \eqref{eq-$c$-biharmonic-intro}, specifically \eqref{eq0-comp-ctau2}, \eqref{eq5-comp-ctau2}, \eqref{eq6-comp-ctau2} and \eqref{eq7-comp-ctau2}, and considering both the tangential and normal components, we obtain the desired result.

\end{proof}

\begin{remark}
It is worth to point out that if $\varphi:M^{m}\hookrightarrow {\mathbb L}^m(\varepsilon)\times \mathbb{R}$ is a totally geodesic immersion, that is if $A\equiv 0$, then, using \eqref{eq:XTheta}, $T(\cos\alpha)=0$. Thus, totally geodesic hypersurfaces in ${\mathbb L}^m(\varepsilon)\times \mathbb{R}$ are $c$-biharmonic since \eqref{BHP1} and  \eqref{BHP2} are identically satisfied.
\end{remark}

\section{Proof of Theorem~\ref{theorem:cbiharmonicconstantcurvatures}}\label{sec-proofTh1}

We first recall a result of R.~Chaves and E.~Santos.

\begin{theorem}[{\cite[Theorem~4.1]{MR4032814}}]\label{teo:chaves-santos}
Let $\varphi:M^{m}\hookrightarrow {\mathbb L}^m(\varepsilon)\times \mathbb{R}$  be a hypersurface with constant principal curvatures such that $T$ is a principal direction and $\cos\alpha\neq 0$ for all $p\in M$. Then $\varepsilon=-1$, $A(T)=0$ and the other $m-1$ principal curvatures are all equal either to $b/\sqrt{1+b^2}$ or to $-b/\sqrt{1+b^2}$ for some real constant $b$.
\end{theorem}

If we assume that the principal curvatures of the hypersurfaces are constant, we have that
\[
H=\frac{k_1+\cdots+k_m}{m}={\rm constant}\quad {\rm and}\quad |A|^2=k_1^2+\cdots+k_m^2={\rm constant}\,.
\]

Now, assuming that both $H$ and $|A|^2$ are constant, the tangential component \eqref{BHP1}  of $\tau_2^c$ simplifies to
\begin{equation}\label{BHP2-1}
\left[\tau_2^c\right]^{\top} =  -\frac{2}{3} \varepsilon\, (m-1)  \cos \alpha \big(3m H\, T  +2 A(T)\big)\,.
\end{equation}
Consequently, if $\varphi:M^{m}\hookrightarrow {\mathbb L}^m(\varepsilon)\times \mathbb{R}$ is $c$-biharmonic, one of the following three cases must occur.\\

\textbf{Case $\varepsilon=0$}.\quad Then ${\mathbb L}^m(\varepsilon)\times \mathbb{R}= \mathbb{R}^{m+1}$ and $\varphi:M^{m}\hookrightarrow \mathbb{R}^{m+1}$ is an isoparametric $c$-biharmonic hypersurface in $\mathbb{R}^{m+1}$. Now, isoparametric hypersurfaces in $\mathbb{R}^{m+1}$ are: totally geodesic hyperplanes; round spheres; cylinders over round spheres. Consequently, for the isoparametric hypersurfaces in $\mathbb{R}^{m+1}$ it is straightforward to check that the vanishing of the normal component \eqref{BHP2} of $\tau_2^c$, that is 
\begin{equation}\label{BHP2-2}
 \frac{5}{3} m H |A|^2 -2 \trace A^3-\frac{2}{3} m^3 H^3=0\,,
\end{equation}
implies that $A=0$. Thus the isoparametric $c$-biharmonic hypersurfaces in $\mathbb{R}^{m+1}$ are only totally geodesic hyperplanes.\\

\textbf{Case $\cos\alpha=0$}.\quad In this case $\partial_t$ is tangent to the hypersurface and $\varphi(M^m)$ is locally a cylinder $M^{m-1}\times\R$ where $M^{m-1}\hookrightarrow {\mathbb L}^m(c)$ is a isoparametric $c$-biharmonic hypersurface in the space form ${\mathbb L}^m(\varepsilon)$. \\

\textbf{Case $\varepsilon\neq 0$ and $ \cos\alpha\neq 0$}.\quad In this case, from \eqref{BHP2-1}, the vector field $T$ satisfies
\begin{equation}\label{eq:Tprincipal}
A(T)=-\frac{3mH}{2}T
\end{equation}
which implies that $T$ is a principal direction with corresponding principal curvature $-3mH/2$. We can thus apply Theorem~\ref{teo:chaves-santos} and since the principal curvature corresponding to $T$ is $0$, we have that $H=0$. Now, the vanishing of the normal component \eqref{BHP2} of $\tau_2^c$ gives
\[
\trace A^3=\pm (m-1)\frac{b^3}{(1+b^2)^{3/2}}=0\,.
\]
Thus $b=0$ and the hypersurface is totally geodesic.

\section{Isoparametric $c$-biharmonic hypersurfaces in $\s^{m+1}$}\label{isoparametric-in-sm}\label{sec-iso-Sn}

In \cite{arXiv:2311.04493} the authors began the study of isoparametric $c$-biharmonic hypersurfaces in space forms. As a premise, they proved the following results:
\begin{proposition}\label{prop-BNO-general}\quad
\begin{enumerate}
\item[{\rm (a)}] Any totally geodesic hypersurface in ${\mathbb L}^{m+1}(\varepsilon)$ is $c$-biharmonic.
\item[{\rm (b)}] Let $M^m$ be a minimal hypersurface in ${\mathbb L}^{m+1}(\varepsilon)$. Then $M^m$ is $c$-biharmonic if and only if its scalar curvature is constant and ${\rm Tr}\,A^3=0$. 
\end{enumerate}
\end{proposition}

From this, they deduced that any hyperplane in $\R^{m+1}$ is $c$-biharmonic, while there is no example of $c$-biharmonic hypersurfaces in $\R^{m+1}$ amongst $m$-dimensional round spheres or cylinders.

Next, they characterized $c$-biharmonic small hyperspheres and generalized Clifford tori in $\s^{m+1}$. Finally, they gave a full analysis of isoparametric $c$-biharmonic hypersurfaces  in the hyperbolic space $\h^{m+1}$.

As an integration of the above work, in this section we investigate the existence of isoparametric $c$-biharmonic  hypersurfaces in $\s^{m+1}$ in the remaining cases, that is, isoparametric hypersurfaces of degree $\ell=3,4,6$.  

We recall that each isoparametric hypersurface in $\s^{m+1}$ is CMC and has $\ell$ distinct constant principal curvatures $k_1>\cdots >k_\ell$ with constant multiplicities $m_1,\ldots,m_\ell$, $m_1+\ldots +m_\ell=m$. Moreover, the only possible values for $\ell$ are $1,2,3,4,6$, and $m_{i+2}=m_i$, so that there are at most two distinct multiplicities, which we shall denote $m_1,m_2$. The integer $\ell$ is called the \textit{degree} of the isoparametric hypersurface. 

We follow the notation of \cite{MOR-Israel}. By a suitable choice of the orientation we can assume $m_1 \leq m_2$ and we have the following explicit description of the principal curvatures of the isoparametric families $M_{s}^m$ in $\s^{m+1}$, $s \in (0,\pi/\ell)$:
\begin{equation}\label{principal-curvatures}
k_i(s)= \cot \left( s +\frac{(i-1)\pi}{\ell} \right ), \quad i=1,\ldots,\ell\,.
\end{equation}
Next, we recall some further useful facts:\\

\textbf{Case $\ell=3$}.\quad There exist only four examples of this type, corresponding to $m_1=m_2=1,2,4,8$.\\

\textbf{Case $\ell=6$}.\quad It was proved in \cite{MR714104} that in this case necessarily $m_1=m_2$ and the possible values are $m_1=1$ or $m_1=2$. \\

\textbf{Case $\ell=4$}.\quad This is the richest case. The only examples with $m_1=m_2$ occur when $m_1=m_2=1$ or $m_1=m_2=2$, but we have plenty of examples with $m_1< m_2$. Indeed, instances occur when $m_1=4$ and $m_2=5$, or $1+m_1+m_2$ is a multiple of $2^{\xi(m_1-1)}$, where $\xi(n)$ denotes the number of natural numbers $p$ such that $1 \leq p \leq n$ and $p \equiv 0,1,2,4 \,({\rm mod}\, 8)$.

To give a more explicit idea of the situation, we point out that the following pairs of multiplicities come about: $m_1=1,m_2\geq 2$; $m_1=2,m_2=2k,k\geq 2$; $m_1=3,m_2=4$; $m_1=4,m_2=4k-5,k\geq 3$.

\vspace{2mm}

Our first result is
\begin{theorem}\label{th-isop-minimal}\quad 
\begin{enumerate}
\item[{\rm (a)}] Let $M^m$ be a minimal isoparametric hypersurface in $\s^{m+1}$ of degree $\ell$, $\ell=3\,\,{\rm or}\,\,6\,$. Then $M^m$ is $c$-biharmonic.
\item[{\rm (b)}] Let $M^m$ be a minimal isoparametric hypersurface in $\s^{m+1}$ of degree $4$ and assume that $m_1=m_2$. Then $M^m$ is $c$-biharmonic.
\end{enumerate}
\end{theorem}
\begin{proof} According to Proposition~\ref{prop-BNO-general}(b) we just have to verify that in our hypotheses the scalar curvature is constant and ${\rm Tr}\,A^3=0$. As for the statement concerning $\Scal$, we recall that
\[
\Scal=m(m-1)+m^2 H^2-|A|^2\,.
\]
Then, since $H$ and $|A|^2$ are constant, so is $\Scal$.  Thus, it remains to compute $\trace A^3$. In both cases $\ell=3,6$ we know that $m_1=m_2$, while if $\ell=4$ this is true by assumption. 

It follows easily that, for each fixed $\ell$, the condition $\trace A^3=0$ becomes equivalent to
\begin{equation}\label{eq-TrA3=0}
\sum_{i=1}^\ell \left (k_i(s) \right)^3=0\,.
\end{equation}
But minimality occurs if and only if $s=\pi / (2 \ell)$ and using \eqref{principal-curvatures} we can handle the three cases:
\[
\begin{array}{ll}
{\rm Case}\,\ell=3 \colon & k_1=\sqrt{3},\,k_2=0,\,k_3=-k_1\\
{\rm Case}\,\ell=4 \colon & k_1=\sqrt{2}+1,\,k_2=\sqrt{2}-1,\,k_3=-k_2,\,k_4=-k_1\\
{\rm Case}\,\ell=6 \colon &k_1=\sqrt{3}+2,\,k_2=1,\,k_3=2-\sqrt{3},\,k_4=-k_3,\,k_5=-k_2,\,k_6=-k_1
\end{array}
\]
from which it is immediate to deduce that \eqref{eq-TrA3=0} holds and so the proof is completed.

\end{proof}

For our purposes, it is now useful to recall that, from \cite{arXiv:2311.04493}, an isoparametric hypersurface  in $\s^{m+1}$ is $c$-biharmonic if and only if

\begin{equation}\label{eq-$c$-biharmonicity-isopar}
[\tau_2^c]^\perp(s)=H m \left(5 |A|^2-2 m^2 H^2-2 m^2+11 m-6\right)-6 \trace A^3=0 \,.
\end{equation}
Since \eqref{eq-$c$-biharmonicity-isopar} represents the vanishing of the normal component of $\tau_2^c$, we have denoted with $[\tau_2^c]^\perp(s)$ the left-hand side of \eqref{eq-$c$-biharmonicity-isopar}.

Therefore, we are now in the right position to perform a more specific analysis of the existence of isoparametric $c$-biharmonic  hypersurfaces of degree $\ell$ in $\s^{m+1}$, $\ell=3,4,6$. 
Indeed, since according to \eqref{principal-curvatures} we know precisely the principal curvatures, we can plug this information into \eqref{eq-$c$-biharmonicity-isopar} and obtain the explicit condition of $c$-biharmonicity in terms of the parameter $s$.\\

We shall obtain a converse of Theorem~\ref{th-isop-minimal}. To this purpose, we proceed case by case. \\

\textbf{Case $\ell=3$}. \quad We know that there exist only four examples of this type, corresponding to $m_1=m_2=1,2,4,8$. The $c$-biharmonicity equation \eqref{eq-$c$-biharmonicity-isopar} becomes
\begin{equation}\label{eq:c-bihar-l3}
\frac{27 m_1 \cos (3 s)  \Big(2 m_1^2+m_1 \cos (6 s)-6 m_1+6\Big)}{\sin^3(s)(2 \cos (2 s)+1)^3}=0\,.
\end{equation}

Now, a routine analysis shows that, on the relevant interval $(0,\pi/3)$, the unique solution of equation \eqref{eq:c-bihar-l3} is $s=\pi/6$ and this value corresponds to the minimal hypersurface.

By way of conclusion, we have shown that the converse of Theorem~\ref{th-isop-minimal}(a) holds in this case, that is:
\begin{proposition}\label{pro-c-biharmonic-iso-l3}
An isoparametric hypersurface in $\s^{m+1}$ of degree $3$ is $c$-biharmonic if and only if it is minimal.
\end{proposition}

\textbf{Case $\ell=6$}.\quad In this case necessarily $m_1=m_2$ and the possible values are $m_1=1$ or $m_1=2$. The analysis proceeds precisely as in the case that $\ell=3$ and so we omit the details. 

Again, the conclusion is: \textit{the minimal isoparametric hypersurfaces in $\s^{m+1}$ of degree $6$ are the only isoparametric $c$-biharmonic hypersurfaces in $\s^{m+1}$ of degree $6$.}\\

\textbf{Case $\ell=4$}.\quad In this case it is convenient to separate the two instances $m_1=m_2$ and $m_1< m_2$.\\

\textbf{Subcase $m_1=m_2$}.\quad The analysis can be carried out as in the cases $\ell=3,6$ and so we omit the details. 
The conclusion is: \textit{when $m_1=m_2$, the minimal isoparametric hypersurfaces in $\s^{m+1}$ of degree $4$ are the only isoparametric $c$-biharmonic hypersurfaces in $\s^{m+1}$ of degree $4$.}\\

\textbf{Subcase $m_1 < m_2$}.\quad This case is interesting because it provides examples of isoparametric $c$-biharmonic hypersurfaces which are \textit{non-minimal}. More precisely, we prove:
\begin{theorem}\label{Th-c-biharm-not-minimal-l=4}
Let $M_s^m$, $s\in(0,\pi/4)$, be a family of isoparametric hypersurfaces of degree $4$ in $\s^{m+1}$ with $m_1<m_2$, $m=2(m_1+m_2)$. Then there exists a unique $s^*\in(0,\pi/4)$ such that $M_{s^*}^m$ is an isoparametric $c$-biharmonic hypersurface in $\s^{m+1}$. Moreover, the $c$-biharmonic hypersurface $M_{s^*}^m$ is non-minimal.
\end{theorem}
\begin{proof}
Direct substitution and simplification yields:
\[
H=\frac{ \Big((m_1+m_2) \cos (4s)+m_1-m_2 \Big)}{\sin(4s) (m_1+m_2)}\,
\]
\[
[\tau_2^c]^\perp=\frac{1}{32 \sin^3(s) \cos^9(s)\left(\tan ^2(s)-1\right)^3} \,A(s)\,,
\]
where
\begin{equation}\label{eq:polyAs}
\begin{aligned}
A(s)=&2 (m_1 - m_2) \big(12 (m_1^2 + m_2^2) + 8 m_1 m_2 - 33 ( m_1 + m_2) + 36\big)\\
   &+ \big(32 (m_1^3 + m_2^3) - 83 (m_1^2 + m_2^2) - 6 m_1 m_2 + 96 (m_1 + m_2)\big)\cos (4 s)\\
   &+2 (m_1 - m_2) \big(4 (m_1^2 + m_2^2) + 8 m_1 m_2 - 7 (m_1 + m_2) + 12\big) \cos (8 s)\\
   &+3 (m_1 + m_2)^2 \cos (12 s)\,.
\end{aligned}
\end{equation}

Now, it is easy to check that there exists a unique value $\tilde{s}\in (0,\pi/4)$ such that $H(\tilde{s})=0$, i.e.,
\[
\tilde{s}= \frac{1}{4} \cos
   ^{-1}\left(\frac{m_2-m_1}{m_1+m_2}\right)\,.
\]
Next, we compute
\begin{equation}\label{eq:limiti}
\begin{aligned}
\lim_{s\to 0^+} A(s) =& 32 m_1 (6 + m_1 (-5 + 2 m_1))>0\\
\lim_{s\to \pi/4^-} A (s)  =&-32 m_2 (6 + m_2 (-5 + 2 m_2))<0\,.
\end{aligned}
\end{equation}
From this it is easy to deduce that $\lim_{s \to 0^+}[\tau_2^c]^\perp(s)=-\infty$ and $\lim_{s \to \pi/4 ^-}[\tau_2^c]^\perp(s)=+\infty$. 

Therefore, there exists $s^*\in (0,\pi/4)$ such that $[\tau_2^c]^\perp(s^*)=0$. 

Finally, 
\[
[\tau_2^c]^\perp(\tilde{s})= \frac{48 (m_1^2-m_2^2)}{\sqrt{m_1 m_2}}   \neq 0
\]
which implies $s^*\neq \tilde{s}$. Therefore $M_{s^*}^m$ is $c$-biharmonic and non-minimal.
To end the proof it remains to show that $s^*$ is the unique solution of $A(s)=0$ in $(0,\pi/4)$. We rewrite $A(s)$ in \eqref{eq:polyAs} as a polynomial by setting $y=\cos(4s)$.  We obtain
\begin{equation}\label{eq:polyAy}
A(y)=a_o+a_1y+a_2y^2+a_3y^3\,,
\end{equation}
where
\[
\begin{aligned}
a_0 &= 4 (m_1 - m_2) \big(4(m_1^2+m_2^2)-13(m_1+m_2)+12\big)\\
a_1 &=  4 \big(8 (m_1^3+ m_2^3) -23(m_1^2+m_2^2)-6m_1 m_2+24(m_1+m_2)\big)\\
a_2 &= 4 (m_1 - m_2) \big(4(m_1^2+m_2^2)+8m_1 m_2-7(m_1+m_2)+12\big)\\
a_3 &= 12 (m_1 + m_2)^2\,.
\end{aligned}
\]
We have to show that the polynomial $A(y)$ admits only one root in $(-1,1)$. We deduce from \eqref{eq:limiti} that  $A(-1)<0$  and $A(1)>0$.
Then, it is enough to show that $A'(y)$ admits at most one zero in $(-1,1)$. It is easy to check that $A''(-1)<0$ and that $A''(y)$ is strictly increasing. Next a straightforward analysis shows that
\[
\begin{cases}  \text{(i)} \quad A''(1)<0 \quad \text{if}\; m_2=m_1+k, \; k\geq 3\\ 
\text{(ii)} \quad A''(1)>0 \quad \text{if}\; m_2=m_1+1 \;\text{or}\; m_2=m_1+2.
\end{cases}
\]
In the case (i), we have $A''(y)<0$ for all $y\in (-1,1)$ and so $A'(y)$ is strictly decreasing on $(-1,1)$. Therefore, $A'(y)$ admits at most one zero in $(-1,1)$.

In the case (ii), there exists a unique point $y_0=-a_2/(3a_3)\in(-1,1)$ such that $A''(y_0)=0$. Then $y_0$ is a minimum of $A'(y)$ and a direct computation shows that $A'(y_0)>0$. We thus conclude that $A'(y)>0$ for all $y\in(-1,1)$.
\end{proof}

 \begin{remark}
We point out that the analysis in the proof of Theorem~\ref{Th-c-biharm-not-minimal-l=4} has shown that \textit{minimal isoparametric hypersurface in $\s^{m+1}$ of degree $\ell=4$ with $m_1< m_2$ are not $c$-biharmonic}.
 \end{remark}
 
 \begin{remark}
In \cite[Theorem 5.1]{MR1868938} compact isoparametric Willmore hypersurfaces in
\(\s^{m+1}\) are characterized. Their structure is similar to isoparametric $c$-biharmonic hypersurfaces in spheres: for \(\ell=2\) they are represented by certain Clifford tori,
while for \(\ell=3,6\) they must be minimal. As in the case of isoparametric $c$-biharmonic hypersurfaces in spheres, \(\ell=4\) represents the richest case and allows for additional non-minimal solutions.
\end{remark}

\begin{remark}
We would like to point that Theorem~\ref{Th-c-biharm-not-minimal-l=4} highlights another interesting
difference between biharmonic and $c$-biharmonic isoparametric hypersurfaces in spheres. It is well-known that there do not exist proper biharmonic isoparametric hypersurfaces in $\s^{m+1}$ of degree \(\ell=4\), see \cite{MR2943022} for the details,
whereas Theorem \ref{Th-c-biharm-not-minimal-l=4} shows that there exist isoparametric $c$-biharmonic  hypersurfaces in this case.
\end{remark}
 
 \section{The proofs of Theorems~\ref{teo2-m=2-remark}, \ref{teo2-remark}, \ref{teo2-m=4-remark} and \ref{teo-global}}\label{sec-new-remark}
 
As a preliminary step we establish the following result.
 
 \begin{theorem}\label{teo1-remark}
Let $\varphi: M^m \hookrightarrow \mathbb{L}^{m+1}(\epsilon)$ be a $c$-biharmonic hypersurface. Assume that $M^m$ has constant scalar curvature and one of the following holds: {\rm (i)} $m=2$; {\rm (ii)} $m=3$; {\rm (iii)} $m=4$ and there are at most three distinct principal curvatures at any point. Then $\varphi: M^m \hookrightarrow \mathbb{L}^{m+1}(\epsilon)$  is isoparametric.
\end{theorem}
\begin{proof} As a general fact we observe that if $\varphi: M^m \hookrightarrow \mathbb{L}^{m+1}(\epsilon)$ is $c$-biharmonic and has constant scalar curvature then,  from \eqref{eq-$c$-biharmonic-intro},  $\tau_2(\varphi)^{\top}=0$, that is $M^m$ is a biconservative hypersurface. Then, from \cite[Equation~(3.3)]{arXiv:2311.04493}, in this case the $c$-biharmonic tension field becomes
\begin{equation}\label{eq:taucbicconstscalar}
\tau_2^c(\varphi)= m \left(-\Delta H+\frac{H}{3}\left(5|A|^2-2 m^2 H^2-\epsilon\left(2 m^2-11 m+6\right)\right)-\frac{2}{m} \trace A^3\right) \eta\,.
\end{equation}

We shall give a proof of the proposition for the three cases (i), (ii) and (iii) separately. \\

Case (i). It is sufficient to prove that the surface $\varphi: M^2 \hookrightarrow \mathbb{L}^3(\epsilon)$ has constant mean curvature. In fact, a surface with constant scalar curvature has constant Gauss curvature and the conditions
\[
\begin{cases}
2 H=k_1+k_2={\rm constant}\\
K=\epsilon+k_1 k_2={\rm constant}
\end{cases}
\]
would imply that $M^2$ is isoparametric, that is $k_1$ and $k_2$ are constant. 
Now, it remains to prove that $H$ is constant. We argue by contradiction assuming that $H$ is non-constant. Then, since $\varphi: M^2 \hookrightarrow \mathbb{L}^3(\epsilon)$ is biconservative, from Theorem~3.1 of \cite{MR3180932} we have that $K=-3 H^2+\epsilon$ in a neighborhood of any point $p\in M^2$ where $\grad H(p)\neq 0$. This is a contradiction and so the proof is complete. \\

Case (ii). If  $\varphi: M^3 \hookrightarrow \mathbb{L}^4(\epsilon)$ is $c$-biharmonic and has constant scalar curvature, we have already observed that  $M^3$ is a biconservative hypersurface. Now, a result of Fu et al. (see \cite[Theorem~1.1]{MR4903596}) implies that $M^3$ is CMC. 
From
$$
\Scal=\epsilon(m-1) m+m^2 H^2-|A|^2=6 \cdot \epsilon+9 H^2-|A|^2
$$
we obtain that $|A|^2$ must be a constant too.
Next, from \eqref{eq:taucbicconstscalar} we deduce that
also $\trace A^3$ is constant.
In conclusion, $H, |A|^2$ and $\trace A^3$ are constant.
Now, let us denote by $k_1, k_2$ and $k_3$ the principal curvature functions. These functions are the roots of the polynomial 
$$
\begin{aligned}
P(k) & =\operatorname{det}\left(A-k I_3\right) \\
& =-\left(k-k_1\right)\left(k-k_2\right)\left(k-k_3\right)\,.
\end{aligned}
$$
A straightforward computation gives
$$
\begin{aligned}
-P(k)= & k^3-\left(k_1+k_2+k_3\right)   k^2+\left(k_1 k_2+k_2 k_3+k_1 k_3\right) k- 
 k_1   k_2   k_3 \\
= & k^3-\left(k_1+k_2+k_3\right)   k^2+\frac{1}{2}\Big\{\left(k_1+k_2+k_3\right)^2-\left(k_1^2+k_2^2+k_3^2\right)\Big\}   k \\
& -\frac{1}{3}  \Big\{k_1^3+k_2^3+k_3^3-\left(k_1+k_2+k_3\right)^3 \\
&+\frac{3}{2}\left(k_1+k_2+k_3\right)\left[\left(k_1+k_2+k_3\right)^2-\left(k_1^2+k_2^2+k_3^2\right)\right]\Big\}\\
=& k^3-3 H  k^2+\frac{1}{2}\left(9  H^2-|A|^2\right) k
 -\frac{1}{3}\Big\{\trace A^3-27  H^3+\frac{9}{2} H \left[9 H^2-|A|^2\right]\Big\}\,.
\end{aligned}
$$
Thus $P(k)$ is a polynomial with constant coefficients and therefore $k_1$, $k_2$ and $k_3$ have to be constant.\\

Case (iii). A result of Fu et al. (see \cite[Theorem~1.4]{MR4903596}) implies that $M^4$ is either CMC or contained in a certain biconservative non-CMC rotational hypersurface with scalar curvature equal to $12$ and two distinct principal curvatures $-k_1=k_2=k_3=k_4$. 

Now, we show by a direct computation that this non-CMC rotational hypersurface is not $c$-biharmonic. This hypersurface can be described explicitly as follows (see \cite[Section~5]{MR4903596}).

Let $\underline{\vartheta}=(\vartheta_1,\vartheta_2,\vartheta_3)$ be a set of local cordinates of $\s^3 \subset \R^4$. Then
\[
\varphi\left (s,\underline{\vartheta} \right )=\left (h_1(s)\underline{\vartheta}, h_2(s),h_3(s)  \right ) \in \R^6\,,
\] 
where $h_1^2+h_2^2+h_3^2 \equiv 1$, $h_1(s)>0$,  and we can assume that $s$ is the arc length for the profile curve, i.e.,
\[
h_1'^2+h_2'^2+h_3'^2 \equiv 1\,.
\]
The associated principal curvatures are
\begin{equation}\label{eq:k1k2wrth1}
k_1=\frac{h_1+h_1''}{\sqrt{1-h_1^2-h_1'^2}}\,,\quad
k_2=k_3=k_4=-\,\frac{\sqrt{1-h_1^2-h_1'^2}}{h_1}\,.
\end{equation}
Moreover, from \cite[Proposition 2.2]{MR4552081} the biconservativity condition becomes $-k_1=k_2=k_3=k_4$ which implies that $h_1$ satisfies the following ODE:
\begin{equation}\label{eq-ODE-h1}
h_1h_1''+h_1'^2+2 h_1^2-1=0\,.
\end{equation}

We observe that \eqref{eq-ODE-h1} admits the following prime integral
\[
h_1^2 h_1'^2+ h_1^4-h_1^2=C\,.
\]
This implies that 
\begin{equation}\label{eq-ODE-h1-integral}
 h_1'^2=\frac{C}{h_1^2}+1-h_1^2\,
\end{equation}
and a routine analysis shows that necessarily $C\in(-1/4, 0)$. Moreover, taking the derivative of \eqref{eq-ODE-h1-integral} we obtain:
\begin{equation}\label{eq-ODE-h1-second}
 h_1''=-\frac{C+h_1^4}{h_1^3}\,.
\end{equation}
Substituting \eqref{eq-ODE-h1-integral} in \eqref{eq:k1k2wrth1} we obtain
\begin{equation}\label{eq:HinC}
H=\frac{1}{2} k_2=-\frac{\sqrt{-C}}{2 h_1^2}\,.
\end{equation}

Next, we observe that
\begin{eqnarray}\label{eq-dati-non-c-biharmonic}
|A|^2&=&4 k_2^2 =16 H^2\,\\\nonumber
{\rm Tr}A^3&=&2 k_2^3 =16 H^3\,\\\nonumber
-\,\Delta H&=&H''+3 \frac{h_1'}{h_1}\,H' \,,
\end{eqnarray}
where for the last equality we have used the fact that the induced metric on $M^4$ is
\[
ds^2+h_1^2 \, g_{\s^3} \,.
\]
Now, inserting all these facts into the $c$-biharmonic tension field \eqref{eq:taucbicconstscalar} we find 
\[
\tau_2^c(\varphi)=4\left(H''+3 \frac{h_1'}{h_1} H'+8 H^3+2 H\right)\eta
\]
which, using \eqref{eq:HinC}, becomes 
\[
\tau_2^c(\varphi)=4\sqrt{-C} \left(\frac{C+h_1^3
   h_1''-h_1^4}{h_1^6}\right)\eta\,.
\]
Finally, using \eqref{eq-ODE-h1-second}, we obtain
\[
\tau_2^c(\varphi)=4\left( -\frac{2 \sqrt{-C}}{h_1^2}\right)\eta=16 H \eta\,
\]
and so it cannot vanish.
\vspace{3mm}

Now we can conclude the proof of the theorem. Indeed, 
it is known that a hypersurface $\varphi: M^m \hookrightarrow {\mathbb L}^{m+1}(\epsilon)$, $m\geq 4$, with constant scalar curvature and constant mean curvature, having at most three distinct principal curvatures, is isoparametric (see \cite{MR4088868} and also \cite{MR4812338}).
\end{proof}  
 
 \begin{proof}[Proof of Theorem~\ref{teo2-m=2-remark}] Let $\varphi:M^{2}\hookrightarrow\s^3$ be a $c$-biharmonic hypersurface with constant scalar curvature. Then, according to Theorem~\ref{teo1-remark}, the surface is isoparametric and for dimensional reasons the possible values of its degree are $\ell=1,2$. The isoparametric $c$-biharmonic surfaces have been classified in \cite{arXiv:2311.04493} and they are (i) and (ii) as in the statement of the theorem.
\end{proof}
 
 \begin{proof}[Proof of Theorem~\ref{teo2-remark}] Let $\varphi:M^{3}\hookrightarrow\s^4$ be a $c$-biharmonic hypersurface with constant scalar curvature. Then, according to Theorem~\ref{teo1-remark}, the hypersurface is isoparametric and for dimensional reasons the possible values of its degree are $\ell=1,2,3$. Now, if $\ell=1,2$ the isoparametric $c$-biharmonic hypersurfaces have been classified in \cite{arXiv:2311.04493} and they are (i) and (ii) as in the statement of the theorem. As for the case $\ell=3$ the conclusion follows readily from Proposition~\ref{pro-c-biharmonic-iso-l3}.
\end{proof}

\begin{proof}[Proof of Theorem~\ref{teo2-m=4-remark}] Let $\varphi:M^{4}\hookrightarrow\s^5$ be a $c$-biharmonic hypersurface with constant scalar curvature and at most three distinct principal curvatures at any point. Then, according to Theorem~\ref{teo1-remark}, the hypersurface is isoparametric and by assumption the possible values of its degree are $\ell=1,2,3$. Now, if $\ell=1,2$ the $c$-biharmonic isoparametric hypersurfaces have been classified in \cite{arXiv:2311.04493} and they are (i) and (ii) as in the statement of the theorem. As for the case $\ell=3$ we recall that the corresponding multiplicities satisfy $m_1=m_2=1,2,4,8$ and, since the dimension of $M$ is $m=4\neq 3 m_1$, $m_1\in\{1,2,4,8\}$, there exists no example of this degree.  
\end{proof}
\begin{proof}[Proof of Theorem~\ref{teo-global}] First, from
\begin{equation}\label{eq:Scal}
\Scal=12+16 H^2-|A|^2
\end{equation}
we deduce that also $|A|^2$ is constant. Now, simple inspection of \eqref{eq:taucbicconstscalar} shows that necessarily ${\rm Tr}A^3$ is constant as well. Therefore, we can apply \cite[Corollary 1.2]{MR4529031}, which we include here for the convenience of the reader using our notations, and conclude that $M^4$ is isoparametric. 

\begin{corollary}\label{Cor-Tang}\cite[Corollary 1.2]{MR4529031} Let $M^m$ ($m>3$) be a compact hypersurface in $\s^{m+1}$ and assume that:
\begin{itemize}
\item[(a)] $\Scal\geq 0$ 
\item[(b)] For any $q=1,\ldots,m-1$ the principal curvatures $k_1, \dots, k_m$ satisfy 
\[
\sum_{i=1}^m k_i^q ={\rm constant}\,.
\] 
\end{itemize}
Then $M^m$ is isoparametric. Moreover, if $M^m$ has $m$ distinct principal curvatures somewhere, then $\Scal= 0$.

\end{corollary}

Thus, it only remains to check that for all the listed examples we have $\Scal\geq 0$. This is just a case by case verification using \eqref{eq:Scal}. Indeed, if $M^4$ is the totally geodesic hypersphere, $\Scal=12$. If $M^4$ is the hypersphere $\s^4\left( \frac{\sqrt 3}{2}\right)$, then
\[
H^2= \cot^2\left(\frac{\pi}{3}\right)=\frac{1}{3}\,\,,\quad |A|^2=4 \cot^2\left(\frac{\pi}{3}\right)=\frac{4}{3}\,.
\]
Thus, in this case $\Scal=16$. 

Next, we examine the generalized Clifford tori. As for the case $\s^2(1/\sqrt 2)\times \s^2(1/\sqrt 2)$ we have:
\[
H=0\,,\quad |A|^2=2 \cot^2 \left( \frac{\pi}{4}\right )+2 \cot^2 \left( \frac{3\pi}{4}\right )=4
\]
and so $\Scal=8$. As for the second torus $\s^2(r_1)\times \s^2(r_2)$ we have: 
\[
|A|^2=2 \,\frac{r_2^2}{r_1^2}+2\, \frac{r_1^2}{r_2^2}=8\,
\]
from which we deduce $\Scal=4+16 H^2 >0$.

In the case of $\s^1(r_1)\times \s^3(r_2)$ we provide an estimate which is sufficient for our purposes. More precisely, using the explicit expression \eqref{eq:P3(T)}, we observe that
\[
P_3(1)=-8 \quad {\rm and}\quad P_3(2)=1
\]
which imply $1<T^*<2$. Next, we estimate 
\[
|A|^2=\frac{r_2^2}{r_1^2}+3\, \frac{r_1^2}{r_2^2}=\frac{1}{T^{*}}+3\,T^{*} \leq 7\,.
\]
From this and \eqref{eq:Scal} it is immediate to conclude that ${\rm Scal}\geq 5+16 H^2 >0$.

Finally, we examine the minimal isoparametric hypersurface of degree $\ell=4$. Using \eqref{principal-curvatures} with $s=\pi/8$ we obtain
\[
H=0\,, \quad |A|^2= \sum_{i=1}^4 \cot^2 \left( \frac{\pi}{8} +\frac{(i-1)\pi}{4} \right )=12\,.
\]
Thus, consistently with the last sentence of Corollary~\ref{Cor-Tang}, $\Scal=0$ in this case and the proof is completed.
\end{proof}

  \section{The proof of Theorem~\ref{teo-rigidity-totally-geodesic}}
Let $\varphi:M^{m}\hookrightarrow {\mathbb L}^m(\varepsilon)\times \mathbb{R}$  be a totally umbilical hypersurface. Then the shape operator satisfies $A=H \Id$ and it follows that $|A|^2=m H^2$ and $\trace A^3=m H^3$. Under these conditions, the tangential and normal components of $\tau_2^c$, given by \eqref{BHP1} and \eqref{BHP2}, become
\begin{equation}\label{eq-cbh-product-umbilical-tangent}
\left[\tau_2^c\right]^{\top}=-\frac{1}{3} m\left(m+8\right) H \grad H-\frac{2}{3}(m-1)(3m+2) \varepsilon H \cos\alpha\, T\,,
\end{equation}
\begin{eqnarray}\label{eq-cbh-product-umbilical-normal}
\left[\tau_2^c\right]^{\perp} &=&\big\{-m \Delta H+\frac{2}{3} m(m-1)(3-m)\varepsilon H+\frac{1}{3}  (m-1)(7 m-12) \varepsilon H\sin^2 \alpha\nonumber\\
&&-\frac{1}{3} m\left(2m^2-5m+6\right) H^3\big\}\eta\,.
\end{eqnarray}

Moreover, under the assumption that $A=H \Id$ the Codazzi equation \eqref{eq:Codazzieq} becomes
\begin{eqnarray}
\label{eq:Codazzieq-umbilical} 
[X(H)+\varepsilon\cos\alpha \langle X,T\rangle] Y-[Y(H) +\varepsilon\cos\alpha(\langle Y,T\rangle]X=0\,.
\end{eqnarray}
Since the dimension of $M$ is $m\geq 2$, for any tangent vector field $X$, choosing $Y$ so that $X,Y$ are linearly independent, we obtain
 \begin{eqnarray}
\label{eq:Codazzieq-umbilical2} 
X(H)+\varepsilon\cos\alpha \langle X,T\rangle =0\,,\quad \forall X\in C(TM)\,.
\end{eqnarray}
Next \eqref{eq:Codazzieq-umbilical2} implies 
 \begin{eqnarray}
\label{eq:Codazzieq-umbilical3} 
\grad H= -\varepsilon\cos\alpha T \,
\end{eqnarray}
that substituted in \eqref{eq-cbh-product-umbilical-tangent} yields 
$$
(5m^2-10m-4) H \grad H=0\,.
$$
Since $5m^2-10m-4$ does not have integer solutions we conclude that $H$ is constant, eventually zero. In particular, from \eqref{eq:Codazzieq-umbilical3}, the possible cases are $\varepsilon=0$, $\cos\alpha=0$ and $T=0$. If $\varepsilon=0$, then the vanishing of  \eqref{eq-cbh-product-umbilical-normal} implies immediately that $H=0$. If $\cos\alpha=0$, then $M^m$ is a cylinder, which is totally umbilical if and only if it is totally geodesic. Finally, if $T$ vanishes, then $M$ is a slice ${\mathbb L}^m(\varepsilon)\times \{t_0\}$ and so it is totally geodesic. \\

%{\bf Declarations}\\
%
%On behalf of all authors, the corresponding author states that there is no conflict of
%interest.
%\bibliographystyle{acm}
%\bibliography{mainbib}

   \end{document}